\documentclass[12pt]{article}

\usepackage[utf8]{inputenc}
\usepackage[british]{babel}
\usepackage{amsmath,amssymb,amsthm}
\usepackage{color}
\usepackage{graphicx}
\usepackage{float}

\def\bbN{{\mathbb{N}}}
\def\bbR{{\mathbb{R}}}
\def\calB{{\mathcal{B}}}
\def\calC{{\mathcal{C}}}
\def\calM{{\mathcal{M}}}
\def\calO{{\mathcal{O}}}
\def\rmd{{\mathrm{d}}}
\def\dt{{\Delta t}}

\def\sn{\operatorname{sn}}

\newtheorem{Rule}{Rule}
\newtheorem*{Rule2bis}{Rule 2'}
\newtheorem*{Rule3bis}{Rule 3'}
\newtheorem{Definition}{Definition}
\newtheorem{Remarque}{Remark}
\newtheorem{Proposition}{Proposition}

\title{Extending nonstandard finite difference schemes rules to systems of nonlinear ODEs with constant coefficients}
\author{Marc E. Songolo\textsuperscript{a}$^{\ast}$\thanks{$^\ast$Corresponding author. Email: marc.songolo@gmail.com} and 
Brigitte Bid\'egaray-Fesquet\textsuperscript{b} \\
\small\textsuperscript{a}Department of Mathematics and Computer Science, \\
\small University of Lubumbashi, 1825 Lubumbashi, Republic Democratic of the Congo; \\
\small ICAM School of Engineering-Nantes campus, 44470 Carquefou, France; \\
\small IETR, Polytech Nantes, University of Nantes, 44306 Nantes, France; \\
\small \textsuperscript{b}Univ. Grenoble Alpes, Grenoble INP$^{\dag}$\footnote{$^{\dag}$Institute of Engineering Univ. Grenoble Alpes}, LJK, 38000 Grenoble, France}
\date{}

\begin{document}
\maketitle

{\sc Abstract.}
In this paper, we present a reformulation of Mickens' rules for nonstandard finite difference (NSFD) scheme to adapt them to systems of ODEs. This leads to exact schemes in the linear case, and also improve the accuracy in the nonlinear case. In the Hamiltonian nonlinear case, it consists in adding correction terms to schemes derived by Mickens. \\


\section{Introduction}

The nonstandard finite difference (NSFD) method was created to overcome the defects of numerical instabilities presented by traditional methods such as Euler, Runge--Kutta, etc. The numerical instabilities indicate that the discrete equations are not able to model correctly the qualitative properties of the solutions to the differential equations \cite{Mic94}. The rules for constructing NSFD models were derived from the construction of exact schemes for certain equations. These rules were then applied to many types of ordinary differential equations (ODE) and to partial derivatives in order to obtain stable schemes which preserve the qualitative properties of the equations. The reference \cite{Pat16} reviews many recent developments and applications of NSFD schemes.

In this paper, we propose to revisit some of Mickens' rules in the light of recent works \cite{Cie13, CR10, CR11, QT18, SBF18, SBF21}. Especially the second rule, that relates to the renormalization of the discretization step-size, does not \textit{a priori} take into account coupling factors between the equations in a differential system. We intend here to extend Mickens' rules to the case of systems of ODEs where we split a linear and a nonlinear part $X'=AX(t)+B(X(t))$.

The paper is organized as follows: we present Mickens' basic rules in Section 2. Section 3 displays the construction of NSFD schemes for systems of ODEs, from a matrix formulation to a derivation of scalar forms. This leads to define two correctors, the effect of which we explore on two examples in Section 4. The discussion then sets out the possible difficulties in developing this strategy and the way to overcome them. 

\section{The Nonstandard Finite Difference Context}

\subsection{Nonstandard Finite Difference Rules}
\label{sec:rules}

In this section, we first give the rules for the construction of NSFD schemes as proposed by Mickens \cite{Mic94}.

\begin{Rule}
\label{rule:1}
The order of the discrete derivatives must be exactly equal to the order of the corresponding derivatives of the differential equations.
\end{Rule}

\begin{Rule}
\label{rule:2}
Denominator functions for the discrete derivatives must, in general, be expressed in terms of more complicated functions of the step-sizes than those conventionally used. 
\end{Rule}

\begin{Rule}
\label{rule:3}
Nonlinear terms must, in general, be modeled non-locally on the computational grid or lattice.
\end{Rule}

\begin{Rule}
\label{rule:4}
Special solutions of the differential equations should also be special (discrete) solutions of the finite difference models.
\end{Rule}

\begin{Rule}
\label{rule:5}
The finite difference equations should not have solutions that do not correspond exactly to solutions of the differential equations.
\end{Rule}

These rules initially apply to single differential equations. Already in \cite{Mic94} the case of Hamiltonian equations treated as systems of two first order equations makes use of a slightly modified version of Rule \ref{rule:2}. Indeed, the derivatives are approximated by
\begin{equation}
\label{eq:rule2}
\frac{\rmd x}{\rmd t} \simeq \frac{x_{k+1}-\psi(\dt)x_{k}}{\phi(\dt)},
\end{equation}
where $\psi$ and $\phi$ are the functions of the step-size $\dt$ and the parameters of the equations. As suggested by Rule \ref{rule:2}, the denominator $\phi(\dt)$ plays the role of the step size and is such that 
\begin{equation*}
\phi(\dt) = \dt + \calO(\dt^2) \text{ as } \dt\to0.
\end{equation*} 
There is an additional function $\psi$, which is not mentioned in Rule \ref{rule:2} and is close to identity, namely
\begin{equation*}
\psi(\dt) = 1 + \calO(\dt^2) \text{ as } \dt\to0.
\end{equation*}

\subsection{Nonstandard, Exact, and Best Finite Difference Schemes}

\begin{Definition}[Nonstandard finite difference scheme, \cite{Mic00}]
A nonstandard finite difference scheme is any discrete representation of a system of differential equations that is constructed based on the above rules.
\end{Definition}

Originally, in \cite{Mic94}, this was the definition of a \textit{best finite difference scheme}. Indeed these rules have been defined to have \textit{exact finite difference schemes} thus avoiding the usual questions about consistency, stability and convergence. This more or less involves that we know exact solutions of the equations (see Rules \ref{rule:4} and \ref{rule:5}), which is of course not the case in general. It is however expected that schemes constructed with these rules would lead to the "best" schemes.

\section{Nonstandard Finite Difference Models}

We address systems of ordinary differential equations where we separate a linear and a nonlinear part:
\begin{equation}
\label{eq:system}
X'=AX(t)+B(X(t)),
\end{equation}
where $t\in[0,T]$, $X(t)\in\bbR^n$, $A\in\calM_{n\times n}(\bbR)$, and $B\in\calC^{0}(\bbR^n,\bbR^n)$.

The analytical solution to system \eqref{eq:system} can be expressed in integral form by
\begin{equation}
\label{eq:integralform}
X(t+\dt) = e^{\dt A} X(t) + \int_t^{t+\dt} e^{(t+\dt-s)A}B(X(s)) \rmd s. 
\end{equation}
To go further in the explicit computations, we approximate $B(X(s))$ on the time interval $[t,t+\dt]$ by a function of $X(t)$ and $X(t+\dt)$:
\begin{equation*}
B(X(s)) \simeq \calB(X(t),X(t+\dt)).
\end{equation*} 
Inserting this in \eqref{eq:integralform}
\begin{align*}
X(t+\dt) & \simeq e^{\dt A} X(t) + \int_t^{t+\dt} e^{(t+\dt-s)A}ds\ \calB(X(t),X(t+\dt))\\
& = e^{\dt A} X(t) + (e^{\dt A}-I)A^{-1} \calB(X(t),X(t+\dt)),
\end{align*}
where $I$ is the identity matrix in $\calM_{n\times n}(\bbR)$.
The simplest numerical method obtained by this formula is the \textit{exponential Euler approximation} for which $\calB(X(t),X(t+\dt))=B(X(t))$:
\begin{equation*}
X_{k+1} =  e^{\dt A} X_k + (e^{\dt A}-I)A^{-1}B(X_k),
\end{equation*}
where $X_k$ is approximating $X(t_k)$, $t_k=k\dt$ for $k\in\bbN$.
This method makes use of a matrix exponential and is hence called \textit{exponential integrator} \cite{HO10}. \\
Such an approximation for the nonlinear part does however not fulfill Rule \ref{rule:3} which advocates for a nonlocal discretization of the nonlinear part. We will therefore prefer the more general form
\begin{equation}
\label{eq:expintegrator}
X_{k+1} =  e^{\dt A} X_k + (e^{\dt A}-I)A^{-1}\calB(X_k,X_{k+1}).
\end{equation}

\subsection{Matrix formulation}

In view of \eqref{eq:expintegrator}, we define the renormalisation matrix $\Phi(\dt) = (e^{\dt A}-I)A^{-1}$, and we can replace the exponential $e^{\dt A}$ by  $I+\Phi(\dt)A$ in \eqref{eq:expintegrator} to obtain
\begin{equation*}
X_{k+1}=\left(I+\Phi(\dt)A\right)X_k+\Phi(\dt)\calB(X_k,X_{k+1}),
\end{equation*}
or equivalently
\begin{equation}
\label{eq:scheme_matrix}
\Phi^{-1}(\dt)\left(X_{k+1}-X_k\right)= AX_k+\calB(X_k,X_{k+1}),
\end{equation}
where the renormalization matrix verifies the property
\begin{equation*}
\Phi(\dt) = \dt I + \calO(\dt^2) \text{ as } \dt\to0.
\end{equation*}

Let us generalize \eqref{eq:rule2} and hence Rule \ref{rule:2} for systems of ordinary differential equations. The scalar functions $\phi$ and $\psi$ are then replaced by matrix-valued functions $\Phi$ and $\Psi$.

\begin{Rule2bis}
The first order derivatives in a nonstandard scheme for a system of ordinary differential equations should be approximated as
\begin{equation*}
\frac{\rmd X}{\rmd t} \simeq \Phi(\dt)^{-1}(X_{k+1}-\Psi(\dt)X_{k}),
\end{equation*}
where
\begin{equation*}
\Phi(\dt) = \dt I + \calO(\dt^2) \text{ as } \dt\to0,
\end{equation*}
and
\begin{equation*}
\Psi(\dt) = I + \calO(\dt^2) \text{ as } \dt\to0.
\end{equation*}
\end{Rule2bis}

Scheme \eqref{eq:scheme_matrix} is nonstandard. In particular, the discretization of the first order derivative corresponds to the above generalized rule, with $\Psi\equiv I$. The fact that we are able to write an exact or only a best scheme depends on the nonlinearity. 

The major drawback of such a scheme is that we have to evaluate the exponential of matrix $\dt A$. This can prove to be an expensive computation \cite{MV03}. This is the reason why we propose in the next section to reformulate scheme \eqref{eq:scheme_matrix} in a scalar way.

\subsection{Scalar formulation}
\label{sec:scalar}

\subsubsection{Construction}

To reformulate scheme \eqref{eq:scheme_matrix}, we consider the Cayley--Hamilton theorem, which implies that the exponential matrix can be rewritten as a finite expansion in powers of $A$: 
\begin{equation}
\label{eq:Expmatrice}
e^{\dt A} = \alpha_0(\dt)I + \alpha_1(\dt)A + \alpha_2(\dt)A^2 + \cdots 
+\alpha_{n-1}(\dt)A^{n-1},
\end{equation}
where $\alpha_0(\dt), \alpha_1(\dt), \dots, \alpha_{n-1}(\dt) \in \bbR$. The construction of these coefficients in the general case can be found in \cite{MV03}.
 
Introducing expansion \eqref{eq:Expmatrice} in the exponential integration scheme \eqref{eq:expintegrator} yields
\begin{align*}
X_{k+1} =\ 
& \alpha_0(\dt) X_k + \alpha_1(\dt) [AX_k + \calB(X_k,X_{k+1})] \\
& + \sum_{j=2}^{n-1} \alpha_j(\dt) A^{j-1} [AX_k + \calB(X_k,X_{k+1})] \\
& + (\alpha_0(\dt)-1)A^{-1}\calB(X_k,X_{k+1}),
\end{align*}
which also reads
\begin{align*}
\frac{X_{k+1}-\alpha_0(\dt)X_k}{\alpha_1(\dt)} =\ 
& [I+R_1(\dt,A)] [AX_k + \calB(X_k,X_{k+1})] \\
& + R_0(\dt,A)\calB(X_k,X_{k+1}),
\end{align*}
where we define the two correction factors
\begin{equation}
\label{eq:R0andR1}
R_0(\dt,A) = \frac{\alpha_0(\dt)-1}{\alpha_1(\dt)}A^{-1}, 
\hspace{5mm}
R_1(\dt,A) = \sum_{j=2}^{n-1}\frac{\alpha_j(\dt)}{\alpha_1(\dt)}A^{j-1}.
\end{equation}
In addition, we also introduce the notion of correction vectors,
\begin{align*}
T_0(\dt,A,X_k,X_{k+1}) & = R_0(\dt,A)\calB(X_k,X_{k+1}), \\
T_1(\dt,A,X_k,X_{k+1}) & = R_1(\dt,A)\left[AX_k + \calB(X_k,X_{k+1})\right], 
\end{align*}
to write the NSFD scheme as
\begin{equation}
\label{eq:scheme_scalar}
\begin{aligned}
\frac{X_{k+1} - \alpha_0(\dt) X_k}{\alpha_1(\dt)} =\  
& AX_k + \calB(X_k,X_{k+1}) \\
& + T_0(\dt,A,X_k,X_{k+1}) + T_1(\dt,A,X_k,X_{k+1}).
\end{aligned}
\end{equation}
With regard to Mickens' second rule, we identify $\psi(\dt)=\alpha_0(\dt)$ and $\phi(\dt)=\alpha_1(\dt)$. 

\begin{Remarque}
\label{rk:1}
If the system dimension is $n=2$, $R_1(\dt,A) = 0$. In the case of a single equation ($n=1$), the above formulation is not valid since $\alpha_1\equiv0$. \\
For linear systems, the correction $T_0(\dt,A,X_k,X_{k+1})$ vanishes. 
\end{Remarque}

\subsubsection{Order estimate}

\begin{Proposition}
\label{prop:alpha}
The coefficients $\alpha_j(\dt)$ occurring in \eqref{eq:Expmatrice} verify 
\begin{equation*}
\alpha_j(\dt) = \frac{\dt^j}{j!} + \calO(\dt^n).
\end{equation*}
\end{Proposition}

\begin{proof}
Let $S_{n-1}(\dt A)$ be the truncated expansion of $\exp(\dt A)$ in terms of powers of $\dt A$:
\begin{equation*}
S_{n-1}(\dt A) = \sum_{j=0}^{n-1} \frac{\dt^j}{j!} A^j.
\end{equation*}
Then 
\begin{align*}
\exp(\dt A) - S_{n-1}(\dt A) 
& = \sum_{k=n}^{+\infty} \frac{\dt^k}{k!} A^k
= \dt^n A^n \sum_{k=0}^{+\infty} \frac{\dt^k}{(n+k)!} A^k, \\
\|\exp(\dt A) - S_{n-1}(\dt A) \| 
& \leq \dt^n \|A\|^n \sum_{k=0}^{+\infty} \frac{\dt^k}{k!} \|A\|^k
= \dt^n \|A\|^n \exp(\dt\|A\|).
\end{align*}
For $\dt\in]0,\dt_0]$, setting $C=\|A\|^n \exp(\dt_0\|A\|)$, 
\begin{equation*}
\|\exp(\dt A) - S_{n-1}(\dt A)\| \leq C \dt^n.
\end{equation*}
The construction in \cite{MV03} is based on the Cayley--Hamilton theorem. Matrix $A^n$ can be written as a finite expansion in lower powers of $A$, defining coefficients $c_j$, $j=0,\dots,n-1$:
\begin{equation*}
A^n = \sum_{j=0}^{n-1} c_j A^j.
\end{equation*}
This allows to define coefficients $\beta_{kj}$, $k\geq0$, $j=0,\dots,n-1$, such that 
\begin{equation*}
A^k = \sum_{j=0}^{n-1} \beta_{kj} A^j,
\end{equation*}
and the $\beta_{kj}$ can be computed iteratively
\begin{equation*}
\beta_{kj} = \begin{cases}
\delta_{kj} & k<n, \\
c_j & k=n, \\
c_0\beta_{k-1,n-1} & k>n, j=0, \\
c_j\beta_{k-1,n-1} + \beta_{k-1,j-1} & k>n, j>0.
\end{cases}
\end{equation*}
Plugging this in the expansion of $\exp(\dt A)$ in terms of powers of $\dt A$, we obtain coefficients $\alpha_j$:
\begin{equation*}
\alpha_j(\dt) 
= \sum_{k=0}^{n-1} \frac{\dt^k}{k!} \beta_{kj} 
+ \frac{\dt^n}{n!} \beta_{nj} 
+ \sum_{k=n+1}^\infty \frac{\dt^k}{k!} \beta_{kj} 
= \frac{\dt^j}{j!} + \frac{\dt^n}{n!} c_j 
+ \sum_{k=n+1}^\infty \frac{\dt^k}{k!} \beta_{kj}. 
\end{equation*}
This implies that at the precision $\calO(\dt^n)$ the expansion \eqref{eq:Expmatrice} is exactly $S_{n-1}(\dt A)$.
\end{proof}

\begin{Proposition}
The correction factors $R_0$ and $R_1$ have the following series expansion
\label{prop:R}
\begin{align*}
R_0(\dt,A) & = \frac{\dt^{n-1}}{n!}(-1)^{n-1}\det(A) A^{-1} + \calO(\dt^n), \\
R_1(\dt,A) & = \sum_{j=2}^{n-1}\frac{\dt^{j-1}}{j!}A^{j-1} + \calO(\dt^{n-1}).
\end{align*}
\end{Proposition}

\begin{proof}
We compute 
\begin{equation*}
\frac{\alpha_0(\dt)-1}{\alpha_1(\dt)} = \frac{\dfrac{\dt^n}{n!}c_0 + \calO(\dt^{n+1})}{\dt + \calO(\dt^n)} = \frac{\dt^{n-1}}{n!}c_0 + \calO(\dt^n).
\end{equation*}
Besides $c_0=(-1)^{n-1}\det(A)$. For $j\geq2$,
\begin{equation*}
\frac{\alpha_j(\dt)}{\alpha_1(\dt)} = \frac{\dfrac{\dt^j}{j!} + \calO(\dt^{n+1})}{\dt + \calO(\dt^n)} = \frac{\dt^{j-1}}{j!} + \calO(\dt^{n-1}).
\end{equation*}
\end{proof}

\subsubsection{Correction of the right-hand side}

In Equation \eqref{eq:scheme_scalar}, Rule 2 in its generalized scalar form \eqref{eq:rule2} is untouched, but now the discretization of the right-hand side is modified.

\begin{Rule3bis}
In an NSFD scheme for the nonlinear system \eqref{eq:system} satisfying the classical Rule 2 \eqref{eq:rule2}, the usual nonlocal discretization of the right-hand side $AX_k+\calB(X_k,X_{k+1})$ should be supplemented with correction terms:
\begin{equation*}
AX_k+\calB(X_k,X_{k+1})+R_{1}(\dt,A)\left[AX_{k}+\calB(X_k,X_{k+1})\right]+R_{0}(\dt,A)\calB(X_k,X_{k+1}),
\end{equation*}
where $R_0$ and $R_1$ are given by \eqref{eq:R0andR1}.
\end{Rule3bis}

\section{Numerical tests}

We will now study the impact of the corrections in various situations. According to Remark \ref{rk:1} we can find contexts where one or the other correction vanishes.

\subsection{Impact of $R_0$}

\subsubsection{A quadratic nonlinear oscillator}

We first study the impact of $R_0$. According to Remark \ref{rk:1}, we therefore consider a system of two differential equations ($n=2$), so that $R_1\equiv0$.  In \cite{Hu06} the quadratic nonlinear differential equation is presented as a good benchmark for numerical schemes 
\begin{equation}
\label{eq:HarmOsci}
\begin{aligned}
& x'' + x + x^2 = 0,\\
& x(0) = x_0>0,\\
& x'(0) = 0.
\end{aligned}
\end{equation}
This equation occurs for example in human eardrum oscillation modeling. It has the significant advantage to have a known exact solution for $x_0<1/2$, namely
\begin{equation*}
x(t) = x_0 + a\sn^2(\omega t,m),
\end{equation*}
where
\begin{align*}
a & = \frac{-12x_0(1+x_0)}{\sqrt{3(1-2x_0)(3+2x_0)}+3(1+2x_0)}, \\
\omega & = \frac12 \sqrt{\frac12 + x_0 + \frac16\sqrt{3(1-2x_0)(3+2x_0)}}, \\
m & = \frac12 + \frac{3(2x_0^2+2x_0-1)}{3+(1+2x_0)\sqrt{3(1-2x_0)(3+2x_0)}}, 
\end{align*}
and $\sn$ is the Jacobi sine fonction. 
We write the second order equation \eqref{eq:HarmOsci} in the Hamiltonian form
\begin{equation}
\label{eq:SystHamil}
\begin{cases}
x' = y, \\
y' = - x - x^2.
\end{cases}
\end{equation}
We thus obtain a system of the form \eqref{eq:system}, with matrices
\begin{equation*}
A = \begin{pmatrix} 0 & 1 \\ -1 & 0 \end{pmatrix},\ 
B(X) = \begin{pmatrix} 0 \\ b(x) \end{pmatrix},
\end{equation*}
where $b(x)=-x^2$ and with initial data $x(0)=x_0$ and $y(0)=0$.

Besides the comparison with an exact solution, two properties can be used to evaluate the quality of numerical methods. First the exact solutions to \eqref{eq:HarmOsci} are periodic with period 
\begin{equation*}
P = 4 \int_0^{\pi/2} \frac{d\theta}{\sqrt{1-m^2\sin^2\theta}}.
\end{equation*}
Second, the differential equation \eqref{eq:HarmOsci} satisfies a conservation law:
\begin{equation}
\label{eq:conservation}
E(t) \equiv \frac12 (x'(t))^2 + \frac12 x(t)^2 + \frac13 x(t)^3 
= \frac12x_0^2 + \frac13 x_0^3.
\end{equation}

Several methods exist to obtain efficient schemes (called \textit{best schemes} by Mickens \cite{Mic94}) for a harmonic oscillator. We can cite the Gautschi type method \cite{HL99}, the exponential integration method \cite{HO10}, the gradient method \cite{Cie11, Cie13, CR10, CR11} and the NSFD method \cite{Mic94, MOR05, MR94}. Here, we are interested in the NSFD method. In \cite{SBF18} the computation of $\alpha_0$ and $\alpha_1$ is made explicit for two-dimensional matrices in term of the eigenvalues, namely $\pm i$ for the matrix involved in \eqref{eq:SystHamil}, leading to $\alpha_0(\dt)=\cos(\dt)$ and $\alpha_1(\dt)=\sin(\dt)$.
This is coherent with the predictions of Proposition \ref{prop:alpha}. Indeed, we have here $c_0=-1$ and $c_1=0$, and this yields $\alpha_0(\dt)=1-\dt^2/2 + \calO(\dt^3)$ and $\alpha_1(\dt)=\dt + \calO(\dt^3)$.

\subsubsection{Mickens' scheme for Hamiltonian systems}

In \cite{Mic94}, Mickens discretized Equation \eqref{eq:SystHamil} as
\begin{equation*}
\begin{cases}
\dfrac{x_{k+1}-\cos(\dt) x_k}{\sin(\dt)} = y_k, \\[5mm]
\dfrac{y_{k+1}-\cos(\dt) y_k}{\sin(\dt)} = - x_k - \left(x_{k+1}\right)^2.
\end{cases}
\end{equation*}

This scheme has the form \eqref{eq:scheme_scalar}, with $\calB(X_k,X_{k+1})=B(X_{k+1})$ and no correction term. Eliminating $y_k$ leads to a discretization of \eqref{eq:HarmOsci}:
\begin{equation*}
\frac{x_{k+1} - 2x_k + x_{k-1}}{\sin^2(\dt)} 
+ \frac{2\left[1-\cos(\dt)\right]x_k}{\sin^2(\dt)} + x_k^2 = 0.
\end{equation*}
The quantity $\dfrac{2\left[1-\cos(\dt)\right]}{\sin^2(\dt)}$ tends to 1 as $\dt\to0$, but we want to have exactly 1 to have an exact computation of the linear term of Equation \eqref{eq:HarmOsci}. We can express everything in terms of trigonometric functions of $\dt/2$ and write
\begin{equation}
\label{eq:HarmOsci_Mickens1}
\frac{x_{k+1} - 2x_k + x_{k-1}}{[2\sin(\dt/2)]^2} 
+ x_k + \cos^2(\dt/2) x_k^2 = 0.
\end{equation}

One of the consequences of the conservation law \eqref{eq:conservation} is that any periodic solution oscillates with a constant amplitude. The Mickens scheme \eqref{eq:HarmOsci_Mickens1} has this property. It suffices to note that it is invariant for any transformation which swaps $x_{k+1}$ and $x_{k-1}$.

However, in the case of a harmonic oscillator, it must also be shown that there is a constant first integral. Moreover it was established in \cite{Mic94} that the discretization of the nonlinear term used in \eqref{eq:HarmOsci_Mickens1} does not make it possible to obtain a constant (discrete) integral. It is then advisable \cite{Mic94} to discretize the non-linear term as
\begin{equation*}
b(x) \approx - x_k\left(\dfrac{x_{k+1}+x_{k-1}}2\right)
\end{equation*}
and the new scheme is
\begin{equation}
\label{eq:HarmOsci_Mickens2}
\frac{x_{k+1} - 2x_k + x_{k-1}}{[2\sin(\dt/2)]^2} 
+ x_k + \cos^2(\dt/2) x_k\dfrac{x_{k+1}+x_{k-1}}2 = 0,
\end{equation}
which conserves all the qualitative properties of the original equation.

This would necessitate to take 
\begin{equation*}
\dfrac{y_{k+1}-\cos(\dt) y_k}{\sin(\dt)} = - x_{k} - x_{k+1}\dfrac{x_{k+2}+x_k}2
\end{equation*}
as a discretization for the second equation of the Hamiltonian form. The system is then highly implicit and uses time $t+2\dt$ when approximating $B(X(s))$ on the time interval $[t,t+\dt]$ in the integral of Equation \eqref{eq:integralform}.

\subsubsection{Adding a correction term}

Let us now see how the NSFD scheme of Section \ref{sec:scalar} reads for Equation \eqref{eq:SystHamil}. Recall that for $n=2$, $R_1(\dt,A) \equiv 0$. We also have already computed $\alpha_0(\dt)$ and $\alpha_1(\dt)$. The scalar NSFD schemes reads
\begin{equation}
\label{eq:SystHamil_exact}
\frac{X_{k+1}-\cos(\dt) X_k}{\sin(\dt)} = AX_k+\calB(X_k,X_{k+1}) 
- \tan(\dt/2) A^{-1}\calB(X_k,X_{k+1}).
\end{equation}

Up to possible choices for $\calB(X_k,X_{k+1})$, and noticing that $-A^{-1}=A$, we find the same scheme as Mickens' but with a correction factor $\tan(\dt/2)A\calB$. We follow the same steps as in the previous paragraph to obtain a scheme for the initial second order equation.

The scalar NSFD scheme for equation \eqref{eq:SystHamil} reads 
\begin{equation*}
\begin{cases}
\dfrac{x_{k+1}-\cos(\dt) x_k}{\sin(\dt)} = y_k + \tan(\dt/2) b(x_k,x_{k+1}), \\[5mm]
\dfrac{y_{k+1}-\cos(\dt) y_k}{\sin(\dt)} = - x_k + b(x_k,x_{k+1}).
\end{cases}
\end{equation*}
Combining these two equations in the same way than for Mickens' scheme first yields
\begin{align*}
\frac{x_{k+1} - 2x_k + x_{k-1}}{\sin^2(\dt)} =\
& \frac{2\left[\cos(\dt)-1\right]x_k}{\sin^2(\dt)} +  b(x_{k-1},x_k) \\
& + \frac{\tan(\dt/2)}{\sin(\dt)} \left[ b(x_k,x_{k+1}) - \cos(\dt)b(x_{k-1},x_k) \right].
\end{align*}
With the same transformation that led to \eqref{eq:HarmOsci_Mickens1}, we find
\begin{equation}
\label{eq:HarmOsci_exact}
\frac{x_{k+1} - 2x_k + x_{k-1}}{[2\sin(\dt/2)]^2} 
+ x_k = \frac12 [b(x_{k-1},x_k) + b(x_k,x_{k+1})].
\end{equation}

The $\cos^2(\dt/2)$ coefficient of Equations \eqref{eq:HarmOsci_Mickens1} or \eqref{eq:HarmOsci_Mickens2} has disappeared. Besides to obtain $-x_k\dfrac{x_{k-1}+x_{k+1}}2$ in the right-hand side, one has simply to choose $b(x_k,x_{k+1})=-x_kx_{k+1}$, which is a nonlocal discretization of the nonlinearity (and therefore complies to Rule \ref{rule:3}) and is only semi-implicit. It only involves the present and the past but not the future, contrarily to what has been observed for Mickens' scheme \eqref{eq:HarmOsci_Mickens2}.

\subsubsection{Numerical results}

We now compare the previous numerical methods for $x_0=0.25$ which lies in the valid interval for initial data, namely $[0,1/2[$. The exact solution is given by $x_0 + a\sn^2(\omega t,m)$ with $a\simeq-0.55$, $\omega\simeq 0.53$, and $m\simeq0.33$. Its time evolution over the time interval $[0,35]$ is displayed in Figure \ref{fig:Kepler_exact}.
\begin{figure}[H]
\centerline{\includegraphics[width=.5\textwidth]{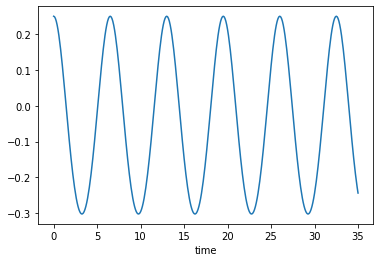}}
\caption{\label{fig:Kepler_exact}Time evolution of the exact solution of Equation \eqref{eq:HarmOsci} for $x_0=0.25$.}
\end{figure}
As already mentioned, this solution is periodic in time.
For the tests we use different values of $\dt$ and compute solutions \textit{via}
\begin{itemize}
\item the explicit Euler scheme,
\item the implicit Euler scheme,
\item Mickens' scheme \eqref{eq:HarmOsci_Mickens1},
\item Mickens' scheme \eqref{eq:HarmOsci_Mickens2},
\item the scalar scheme with correction $R_0$ \eqref{eq:HarmOsci_exact} with $b(x_k,x_{k+1})=-x_kx_{k+1}$.
\end{itemize}
The advantage of comparing the methods with a quadratic nonlinearity is that we are able to compute explicitly the iterates for all these methods without adding methods to solve nonlinear systems such as predictor--correctors or fixed points. The time evolution of the relative error between the exact solution and the computed solutions are shown in Figure \ref{fig:Kepler_error} for different values of the time-step.

Of course, the three NSFD schemes, \eqref{eq:HarmOsci_Mickens1}, \eqref{eq:HarmOsci_Mickens2}, and \eqref{eq:HarmOsci_exact}, outperform the Euler schemes, in particular they do not show a deterioration of the error as time evolves.

The correction $R_0$ does indeed improve Mickens' original schemes, gaining more than two errors of magnitude for small $\dt$. For large $\dt$ the gain is not so clear, but this is due to the approximation of the nonlinearity, which is the only source of approximation in \eqref{eq:HarmOsci_exact}, and which is dominant for $\dt=0.05$.

\begin{figure}[H]
\begin{tabular}{cc}
\includegraphics[width=.5\textwidth]{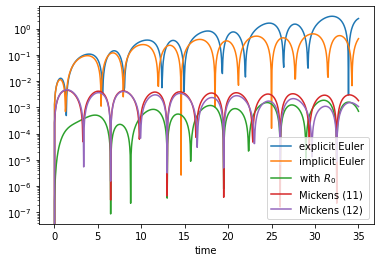}
& \includegraphics[width=.5\textwidth]{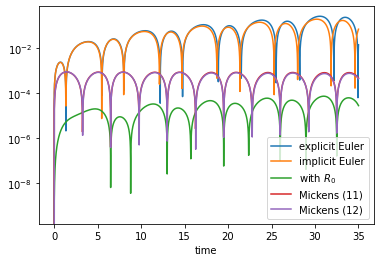} \\
$\dt = 0.05$ & $\dt = 0.01$ \\
\includegraphics[width=.5\textwidth]{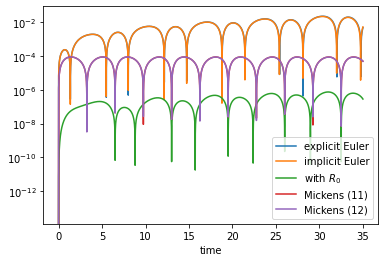}
& \includegraphics[width=.5\textwidth]{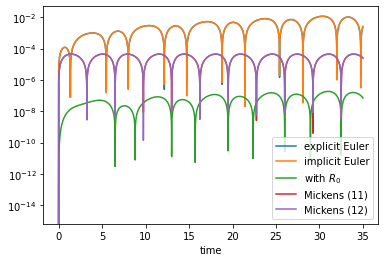} \\
$\dt = 0.001$ & $\dt = 0.0005$ \\
\end{tabular}
\caption{\label{fig:Kepler_error}Time evolution of the relative errors for $\dt=0.05$, $0.01$, $0.001$, and $0.0005$ for Equation \eqref{eq:HarmOsci} and $x_0=0.25$.}
\end{figure}

\subsection{Impact of $R_1$}
\label{sec:biomass}

\subsubsection{A forest biomass model}

To test the impact of $R_1$ only, we now consider a linear system with dimension greater than 2, namely $n=3$. Again, according to Remark \ref{rk:1}, $T_0\equiv0$ in this linear example. We use a simple example found in \cite{GW98} dealing with the evolution forest biomass. More precisely, we denote $x(t)$ the biomass decayed into humus, $y(t)$ the biomass of dead trees, and $z(t)$ the biomass of living trees. The corresponding evolution equations are
\begin{equation}
\label{eq:Biomass}
\begin{cases}
x'(t) = - x(t) + 3y(t), \\
y'(t) = - 3y(t) + 5z(t), \\
z'(t) = - 5z(t),
\end{cases}
\end{equation}
with an initial data where there are no dead trees and no humus at $t=0$, namely $x(0)=0$, $y(0)=0$, and $z(0)=z_0$. The corresponding matrix $A$ is
\begin{equation}
\label{eq:Biomass_matrice}
A = \begin{pmatrix} -1 & 3 & 0 \\ 0 & -3 & 5 \\ 0 & 0 & -5 \end{pmatrix},
\end{equation}
yielding the exact solution
\begin{equation*}
\begin{cases}
x(t) = \dfrac{15}8 \left(e^{-t}-2e^{-3t}+e^{-5t}\right) z_0, \\ 
y(t) = \dfrac52 \left(e^{-3t}-e^{-5t}\right) z_0, \\ 
z(t) = e^{-5t} z_0. 
\end{cases}
\end{equation*}

\subsubsection{Derivation of correction terms}

For matrix $A$ given by Equation \eqref{eq:Biomass_matrice}, we have $A^3=-15-23A-9A^2$, i.e. $c_0=-15$, $c_1=-23$, and $c_2=-9$. We therefore predict that $\alpha_0(\dt) = 1 - \frac52 \dt^3 + \calO(\dt^4)$, $\alpha_1(\dt) = \dt - \frac{23}6 \dt^3 + \calO(\dt^4)$, and $\alpha_2(\dt) = \frac12 \dt^2 - \frac32 \dt^3 + \calO(\dt^4)$.
Writing
\begin{equation*}
(\alpha_0(\dt)I+\alpha_1(\dt)A+\alpha_2(\dt)A^2)\begin{pmatrix} 0 \\ 0 \\ z_0\end{pmatrix} 
= \begin{pmatrix} x(\dt) \\ y(\dt) \\ z(\dt)\end{pmatrix}
\end{equation*}
yields
\begin{align*}
\alpha_0(\dt) & = \frac{15}8 e^{-\dt} - \frac54 e^{-3\dt} + \frac38 e^{-5\dt},\\
\alpha_1(\dt) & = e^{-\dt} - \frac32 e^{-3\dt} + \frac12 e^{-5\dt},\\
\alpha_2(\dt) & = \frac18 e^{-\dt} - \frac14 e^{-3\dt} + \frac18 e^{-5\dt}.
\end{align*}
These values do agree with the predicted expansions at order 3. Since there is no nonlinear part, 
\begin{align}
X_{k+1} 
& = \alpha_0(\dt) X_k + \alpha_1(\dt) A X_k + \alpha_1(\dt) T_1(\dt, A, X_k) \nonumber \\
\label{eq:Biomass_NSFD}
& = \alpha_0(\dt) X_k + \alpha_1(\dt) A X_k + \alpha_2(\dt) A^2 X_k.
\end{align}

\subsubsection{Numerical results}

For the numerical test case, we compare our method \eqref{eq:Biomass_NSFD}, which should be exact since no approximation has been done in its derivation, with the Euler explicit and implicit methods. We also compute  
\begin{align*}
X_{k+1} = \gamma_0(\dt) X_k + \gamma_1(\dt) A X_k + \gamma_2(\dt) A^2 X_k,
\end{align*}
where the $\gamma_j$ are the order 3 approximations of $\alpha_j$, namely $\gamma_0(\dt) = 1 - \frac52 \dt^3$, $\gamma_1(\dt) = \dt - \frac{23}6 \dt^3$, and $\gamma_2(\dt) = \frac12 \dt^2 - \frac32 \dt^3$.
Finally we derive a NSFD scheme on the above principles but for each equation separately, leading to
\begin{equation}
\label{eq:Biomass_Trad_NSFD}
\begin{cases}
\dfrac{x_{k+1} - x_k}{1-e^{-\Delta t}} = - x_k + 3 y_k, \\
\dfrac{y_{k+1} - y_k}{(1-e^{-3\Delta t})/3} = -3 y_k + 5 z_k, \\
\dfrac{z_{k+1} - z_k}{(1-e^{-5\Delta t})/5} = - 5 z_k.
\end{cases}
\end{equation}

The exact solution is computed over $[0,10]$, corresponding to ten years of time evolution, and the result is displayed on Figure \ref{fig:Biomass_exact}. 
\begin{figure}[H]
\centerline{\includegraphics[width=.5\textwidth]{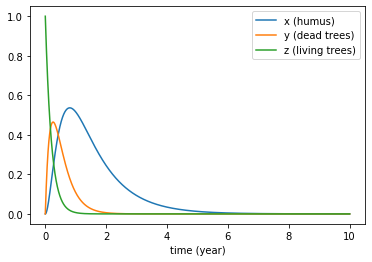}}
\caption{\label{fig:Biomass_exact}Time evolution of the exact solution of Equation \eqref{eq:Biomass} for $z_0=1$.}
\end{figure}
As expected from the equations, the biomass decayed into humus (corresponding to $x$) has the slowest time evolution and the errors accumulated on $x$ are greater than on the other variables. This is why we will show the errors on this variable. Dealing with a variable that is naturally decaying to zero the relative errors are computed as
\begin{equation*}
E_k = \frac{|x_k-x^e_k|}{x^e_k}, 
\end{equation*} 
where $x^e_k$ is the exact value and $x_k$ the computed value. Figure \ref{fig:Biomass_error} shows the relative errors for the different methods.

\begin{figure}[H]
\begin{center}
\begin{tabular}{ccc}
\includegraphics[width=.5\textwidth]{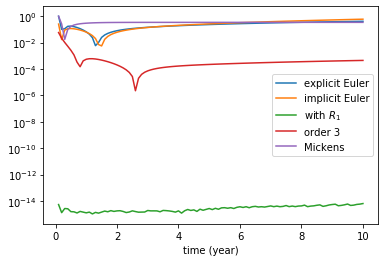}
& \includegraphics[width=.5\textwidth]{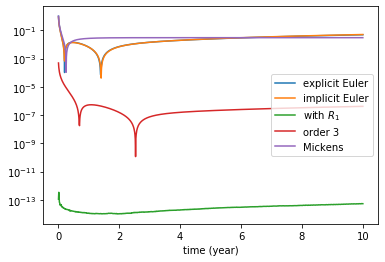} \\
$\dt = 0.1$ & $\dt = 0.01$ 
\end{tabular}
\includegraphics[width=.5\textwidth]{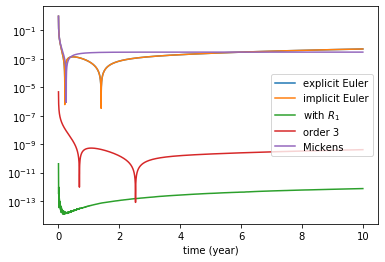} \\
$\dt = 0.001$ 
\end{center}
\caption{\label{fig:Biomass_error}Time evolution of the relative errors for $\dt=0.1$, $0.01$, $0.001$ for the forest biomass model.}
\end{figure}

As expected our method is exact. The order 3 method also behaves very well. It has the major advantage to be derived only with the knowledge of the coefficient of the characteristic polynomial of matrix $A$ which is much easier to compute than the $\alpha_j$. The performance of the traditional NSFD method \eqref{eq:Biomass_Trad_NSFD} is comparable to that of the explicit and implicit Euler methods.

\subsection{Impact of $R_0$ and $R_1$}

\subsubsection{A forest biomass model with constant force}

To test the impact of both $R_0$ and $R_1$, we consider the forest biomass model \eqref{eq:Biomass}, in which we introduce a constant forcing by planting trees. This corresponds to add a constant $z_f$ in the right-hand side of the last equation, modeling the time evolution of living trees. Hence the system reads
\begin{equation}
\label{eq:Biomass_trees}
\begin{cases}
x'(t) = - x(t) + 3y(t), \\
y'(t) = - 3y(t) + 5z(t), \\
z'(t) = - 5z(t) + z_f,
\end{cases}
\end{equation}
with initial conditions $x(0)=0$, $y(0)=0$, and $z(0)=z_0$.

The analytical solution is given by
\begin{equation*}
\begin{cases}
x(t) = \dfrac{15}8 \left(e^{-t} - 2 e^{-3t} + e^{-5t}\right) z_0
+ \dfrac18 \left(8 - 15 e^{-t} + 10 e^{-3t} - 3 e^{-5t} \right) z_f, \\ 
y(t) = \dfrac52 \left(e^{-3t} - e^{-5t}\right) z_0 
+ \dfrac16 \left(2 - 5 e^{-3t} + 3 e^{-5t} \right) z_f, \\ 
z(t) = e^{-5t} (z_0 - \frac{z_f}5) + \frac{z_f}5. 
\end{cases}
\end{equation*}
We display in Figure \ref{fig:Trees} the time evolution of this analytical solution for $z_0=1$ and $z_f=0.5$. We observe in particular the theoretical long time limits,  $z_f$, $z_f/3$, and $z_f/5$ for $x$, $y$, and $z$ respectively. 

\begin{figure}[H]
\centerline{\includegraphics[width=.5\textwidth]{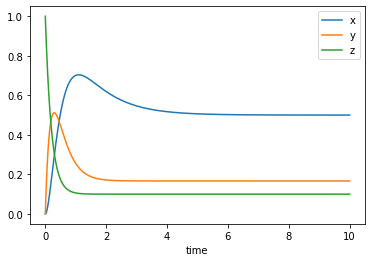}}
\caption{\label{fig:Trees}Time evolution of the exact solution of System \eqref{eq:Biomass_trees} for $z_0=1$ and $z_f=0.5$.}
\end{figure}
The NSFD scheme \eqref{eq:scheme_scalar} reads
\begin{equation}
\label{eq:Trees_NSFD}
\begin{aligned}
X_{k+1} = &\ \alpha_0(\dt) X_k + \alpha_1(\dt) \left[A X_k + \calB\right] 
+ \alpha_2(\dt)A\left[A X_k + \calB\right] \\
&+ (\alpha_0(\dt) - 1)A^{-1}\calB,
\end{aligned}
\end{equation}
where the matrix $A$ and coefficients $\alpha_j$ are the same as in \eqref{eq:Biomass_NSFD}, but now we have a constant nonlinearity $\calB$
\begin{equation}
\calB = \begin{pmatrix} 0 \\ 0 \\ z_f\end{pmatrix}.
\end{equation}
This test case enables to study the impact of the correction term for the nonlinear part without any approximation on the nonlinearity itself. \\
We anew compare this method, with the explicit and implicit Euler schemes, the scheme where the coefficients $\alpha_j$ are replaced by the corresponding $\gamma_j$,
and the traditional NSFD scheme where we replace the last equation in System \eqref{eq:Biomass_Trad_NSFD} by
\begin{equation*}
\dfrac{z_{k+1} - z_k}{(1-e^{-5\Delta t})/5} = - 5 z_k + z_f.
\end{equation*}

\begin{figure}[H]
\begin{center}
\begin{tabular}{ccc}
\includegraphics[width=.5\textwidth]{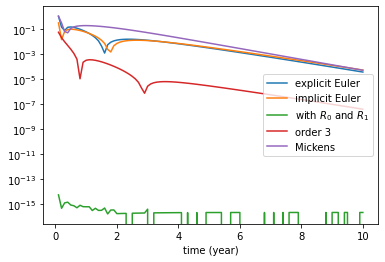}
& \includegraphics[width=.5\textwidth]{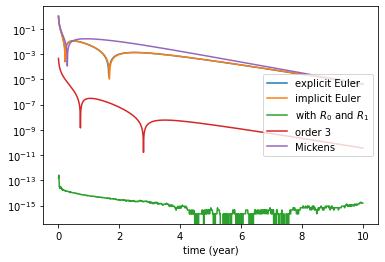} \\
$\dt = 0.1$ & $\dt = 0.01$ 
\end{tabular}
\includegraphics[width=.5\textwidth]{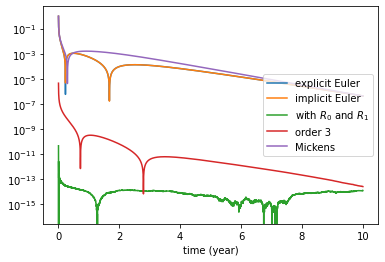} \\
$\dt = 0.001$ 
\end{center}
\caption{\label{fig:Trees_error}Time evolution of the relative errors for $\dt=0.1$, $0.01$, $0.001$ for the forest biomass model with constant forcing.}
\end{figure}

The time evolution of the relative error between the analytical exact solution and the approximated solutions is shown in Figure \ref{fig:Trees_error} for different values of the time step. Our method is exact and behaves very well for any time step. Replacing the $\alpha_j$'s by their third order approximation also yields good results, while the traditional Mickens' NSFD scheme is comparable to the explicit and implicit Euler methods.

\subsubsection{A forest biomass model with a seasonal plantation}

To continue to explore method errors, we now modify the biomass model \eqref{eq:Biomass} to have both corrections $R_0$ and $R_1$ and this time a nonlinearity that models seasonal plantations, and which amounts to performing a sinusoidal forcing
\begin{equation}
\label{eq:Biomass_cosinus}
\begin{cases}
x'(t) = - x(t) + 3y(t), \\
y'(t) = - 3y(t) + 5z(t), \\
z'(t) = - 5z(t) + z_f\left[1 + \cos(\omega t)\right],
\end{cases}
\end{equation}
with initial conditions $x(0)=0$, $y(0)=0$, and $z(0)=z_0$. Such a time dependent forcing will have to be approximated in the numerical schemes.

The exact analytical solution of the new system is
\begin{equation*}
\begin{cases}
x(t) = & \dfrac{15}8 \left(e^{-t} - 2 e^{-3t} + e^{-5t}\right) z_0
+ \dfrac18 \left(8 - 15 e^{-t} + 10 e^{-3t} - 3 e^{-5t} \right) z_f\\
& + 15 \dfrac{3 (5-3\omega^2) \cos(\omega t) + \omega (23-\omega^2) \sin(\omega t)}
               {(1+\omega^2)(9+\omega^2)(25+\omega^2)} z_f \\
& + \dfrac{15}8 \left(\dfrac{- e^{-t}}{1+\omega^2} + \dfrac{6 e^{-3t}}{9+\omega^2} + \dfrac{- 5 e^{-5t}}{25+\omega^2} \right)z_f, \\ 
y(t) = & \dfrac52 \left(e^{-3t} - e^{-5t}\right) z_0 
+ \dfrac16 \left(2 - 5 e^{-3t} + 3 e^{-5t} \right) z_f \\
& + 5 \dfrac{(15-\omega^2) \cos(\omega t) + 8 \omega \sin(\omega t)}{(9+\omega^2)(25+\omega^2)} z_f
+  \dfrac52 \left(\dfrac{- 3 e^{-3t}}{9+\omega^2} +\dfrac{5 e^{-5t}}{25+\omega^2} \right)z_f,\\
z(t) = & e^{-5t} z_0 + \dfrac15 \left(1 - e^{-5t}\right) z_f 
+ \dfrac{5 \cos(\omega t) + \omega \sin(\omega t) } {25+\omega^2} z_f
+ \dfrac{- 5 e^{-5t}} {25+\omega^2} z_f.
\end{cases}
\end{equation*}

This solution is displayed in Figure \ref{fig:Biomass_cosinus}. We choose $z_0=1$ and $z_f=0.5$ to have the same mean limits as in the previous simulations. We also choose $\omega=2\pi$ to have a one year period for the forcing.

\begin{figure}[H]
\centerline{\includegraphics[width=.5\textwidth]{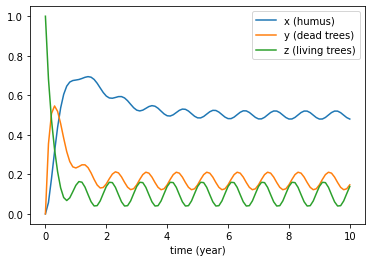}}
\caption{\label{fig:Biomass_cosinus}Time evolution of the exact solution of System \eqref{eq:Biomass_cosinus} for $z_0=1$, $z_f=0.5$ and $\omega=2\pi$.}
\end{figure}

The numerical schemes we compare are exactly the same as before, except for the treatment of $B$:
\begin{equation*}
B(t) = \begin{pmatrix} 0 \\ 0 \\ z_f \left[1 + \cos(\omega t)\right]\end{pmatrix},
\end{equation*}
which is now time-dependent and for which we have to choose an approximation. For the computation of $X_{k+1}$ from $X_k$, five approximations have been used and compared if relevant, namely
\begin{equation*}
\calB_{\rm left} = B(t_k),\
\calB_{\rm right} = B(t_{k+1}),\
\calB_{\rm middle} = B((t_k+t_{k+1})/2),
\end{equation*}
\begin{equation*}
\calB_{\rm half} = (B(t_k) + B(t_{k+1}))/2,\
\calB_{\rm mean} = \int_{t_k}^{t_{k+1}} B(t)dt.
\end{equation*}
The explicit and implicit Euler methods clearly use $\calB_{\rm left}$ and $\calB_{\rm right}$ respectively, but the question is open for the other numerical methods. In a first row of numerical tests we compare the errors when  $\calB$ is approximated by $\calB_{\rm half}$. We choose this because it is the form which (besides the explicit one) is the easiest to extend when nonlinearities involving $X_k$ are concerned. Figure \ref{fig:Biomass_cosinus_error} shows the errors for the five studied schemes. Again our method and its third order approximation outperform the three other schemes.

\begin{figure}[H]
\begin{center}
\begin{tabular}{ccc}
\includegraphics[width=.5\textwidth]{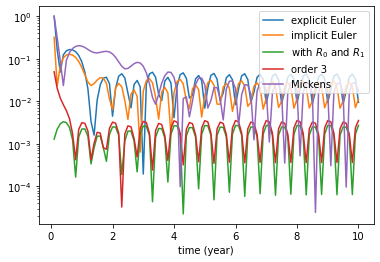}
& \includegraphics[width=.5\textwidth]{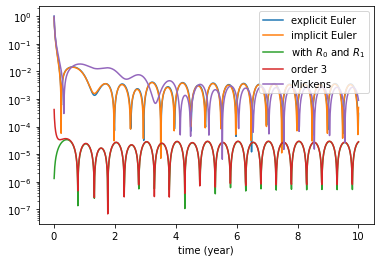} \\
$\dt = 0.1$ & $\dt = 0.01$ 
\end{tabular}
\includegraphics[width=.5\textwidth]{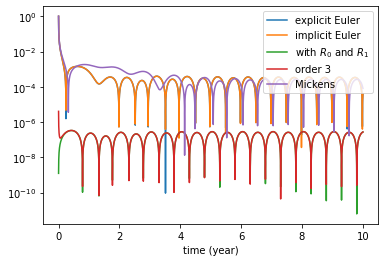} \\
$\dt = 0.001$ 
\end{center}
\caption{\label{fig:Biomass_cosinus_error}Time evolution of the relative errors for $\dt=0.1$, $0.01$, $0.001$ for the forest biomass model with time-dependent forcing.}
\end{figure}

Now we want to discuss the comparison of the two best methods. Since there is an approximation in the estimation of the time-dependent term, and that this approximation is coarser than the approximation in the third order method, both schemes yield very similar results, and the errors are $O(\dt^2)$. We can discuss a little further by comparing the use of $\calB_{\rm left}$, $\calB_{\rm middle}$, $\calB_{\rm half}$, and $\calB_{\rm mean}$ for $\dt=0.001$. The numerical results are displayed in Figure \ref{fig:Biomass_cosinus_errors}. The computation with $\calB_{\rm half}$ yields the worst results among the other methods but the difference is not significant enough to be worth when dealing with more complex nonlinearities or time dependent forcings.

\begin{figure}[H]
\begin{center}
\begin{tabular}{ccc}
\includegraphics[width=.5\textwidth]{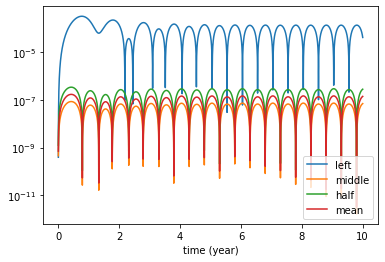}
& \includegraphics[width=.5\textwidth]{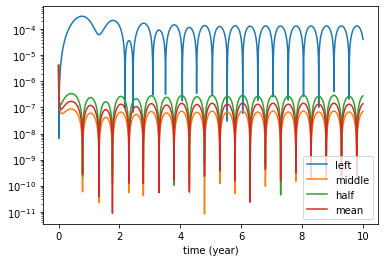} \\
with $\alpha_j$ coefficients & with $\gamma_j$ coeficients
\end{tabular}
\end{center}
\caption{\label{fig:Biomass_cosinus_errors}Comparison of the approximations of the time-dependent forcing for our methods and its third order approximation.}
\end{figure}

\section{Discussion}

\subsection{Extended rules for systems}

We have defined two new rules for NSFD schemes for systems of ODEs. These rules stem from a careful derivation when splitting the equation into a linear and a nonlinear part. The only approximations are made on the nonlinear part. 

In a first step a matrix formulation is given, leading to a generalization of the second rule (Rule 2'), which addresses the treatment of the first derivative. The usual scalar functions $\phi$ and $\psi$, involved in the denominator and the numerator respectively, are then replaced by matrix valued functions. The system is treated as a whole, contrarily to what can usually been done where each equation is taken into account more or less separately. An example of this separate treatment is illustrated by \eqref{eq:Biomass_Trad_NSFD}.

The matrix formulation is an exponential integrator, and deriving a scalar version of this scheme allows to avoid the possible difficulties in computing the matrix exponentials. This leads to usual scalar coefficients in the discretization of the first order derivative, but they are the same for all the equations, and to correction terms in the right-hand side, which are described by Rule 3'. 

In the examples we have separated the effect of the two correction terms on purpose. But of course they are designed also to work together. If the system is linear, or the nonlinearity is a constant forcing term, no approximation is made at any stage of the derivation and the obtained scheme is exact. In the case of a constant forcing term and for at least three coupled equations the two correction terms are nonzero.
 
\subsection{Deriving the scalar coefficients}

The derivation of the scalar coefficient is tedious. The examples shown here are quite simple since they deal with very few equations. In our second example, we computed $\exp(\dt A)$ formally and wrote equation \eqref{eq:Expmatrice}, which led to solve a three-dimensional linear system in the $\alpha_j$. Computing $\exp(\dt A)$ formally needs to know the eigenvalues and eigenvectors. 

Replacing this formal derivation by a numerical determination of the $\alpha_j$, computing $\exp(\dt A)$ numerically and solving the resulting systems numerically can destroy the quality of the method. We have experienced ourselves that even not being careful with the computation of the (scalar) exponentials in the construction of the $\alpha_j$ in Section \ref{sec:biomass} leads to destroy the fine equilibrium that leads to the expansions in Proposition \ref{prop:alpha} and to a not better scheme than the explicit Euler scheme!

If the formal computation is not possible, we strongly recommend to replace the $\alpha_j$ by their $n$-th order approximation as done in Section \ref{sec:biomass} with the "order 3" scheme. This approximation has the advantage to only use the knowledge of the coefficients of the characteristic polynomial. This polynomial is easier to compute than the $\alpha_j$. It is indeed the first step in the computation of the $\alpha_j$. 
Taking $\gamma_j$ simply consists in using the truncated series $S_{n-1}(\dt A) = \sum_{j=0}^{n-1} \frac{\dt^j}{j!} A^j$ instead of the matrix exponential. For a linear system with $n=5$, this is equivalent to use the classical order 4 Runge--Kutta method. For other system dimensions, we also have a Runge--Kutta-like method, but with an order that is adapted to $n$.
 
\subsection{Singular linear part}

In the previous discussion, we have used $A^{-1}$ and implicitly have supposed that $A$ was non-singular. If $A$ is singular, the nonlinearity $B$ can be written as $B=AC+K$ where $K$ belongs to  the kernel of $A$. Then 
\begin{equation*}
\int_0^\dt e^{(\dt-s)A} ds B = (e^{\dt A}-I) C =   (e^{\dt A}-I) A^+ B, 
\end{equation*}   
where $A^+$ is the generalized inverse of $A$. This allows to generalize our approach in the singular case.

\section{Conclusion}

Having considered the NSFD method as a special class of exponential integrators, we have been able to revisit Mickens's rules to apply to systems of ODEs. When these systems are linear, the method is exact. In the Hamiltonian nonlinear case, it consists in adding to Mickens' schemes a correction term, that has been shown to improve the accuracy.

\end{document}